\documentclass[12pt,reqno]{article}

\usepackage[usenames]{color}
\usepackage{color}
\usepackage{amssymb}
\usepackage{graphicx}
\usepackage{amscd}
\usepackage[mathscr]{euscript}
\usepackage[colorlinks=true,
linkcolor=webgreen, filecolor=webbrown,
citecolor=webgreen]{hyperref}

\definecolor{webgreen}{rgb}{0,.5,0}
\definecolor{webbrown}{rgb}{.6,0,0}

\usepackage{color}
\usepackage{fullpage}
\usepackage{float}

\usepackage{graphics,amsmath}
\usepackage{amsfonts,amsthm}
\usepackage{latexsym}

\newtheorem{corollary}{Corollary}[section]

\newcommand{\seqnum}[1]{\href{http://www.research.att.com/cgi-bin/access.cgi/as/~njas/sequences/eisA.cgi?Anum=#1}
{\underline{#1}}}

\begin{document}
\begin{center}
\vskip 1cm{\LARGE\bf Redundant generating functions\\ in lattice
path enumeration } \vskip 1cm
\large Jong Hyun Kim \\
Vaughn College of Aeronautics and Technology \\
Flushing, NY 11369 \\
USA \\
\href{mailto:jong.kim@vaughn.edu}{\tt jong.kim@vaughn.edu} \\
\end{center}

\vskip .2 in
\begin{abstract}
A redundant generating function is a generating function having
terms which are not part of the solution of the original problem.
We use redundant generating functions to study two path problems.
In the first application we explain a surprising occurrence of
Catalan numbers in counting paths that stay below the line $y=2x$.
In the second application we prove a conjecture of
Niederhausen and Sullivan.
\end{abstract}

\newtheorem{theorem}{Theorem}[section]
\newtheorem{lemma}{Lemma}[section]

\section{Introduction}
One method for solving recurrences with boundary conditions is
extending the region in which the recurrence is satisfied. Then
the generating function for the extended recurrence in the new
region may become simpler. Following MacMahon \cite[pp.~128]{MacMahon}, we
call the generating function a redundant generating function,
since it contains some terms which are not part of the solution of
the original problem. This method is especially useful in lattice
path counting problems. In this paper we first review the redundant
generating function for ballot numbers. Then we study variations of the ballot problem. In one of these variations, we explain a surprising occurrence of Catalan numbers. When we define the number $D'_2(m-n,n)$ for $m,n \ge 0$ by
    \begin{equation}\label{eq:200}
    \sum_{m, n = 0}^\infty  D'_2(m-n, n)x^{ m } t^{ n }
    =\Big( 1+ \sum_{n=0}^\infty (-1)^{n+1} C_n t^{n+1}\Big) \Big( 1-x(1+t) \Big)^{-1} ,
    \end{equation}
where $C_n$ is the $n$th Catalan number \cite[pp.~219--229]{EC2} (\seqnum{A000108}), we will see that $D'_2(m,n)$ is the number of lattice paths from $(0,0)$ to $(m,n)$ (where $0 \le n \le 2m$), with unit up $(\uparrow)$ steps
$(0,1)$ and unit right $(\rightarrow)$ steps $(1,0)$, that never
cross the line $y=2x$. There is no combinatorial significance for $D'_2(m,n)$ where $2m < n$ or $n<0$.

Next, we prove a conjecture \cite{NS} of Niederhausen and Sullivan using a redundant generating function. When we define the number $S'(m,n)$ for $m,n \ge 0$ by
   \begin{equation} \label{eq:199}
    \sum_{m, n = 0}^\infty  S'(m, n)x^{ m } t^{ n }
    =\Big( \frac{3+t-\sqrt{(1+t)^2+4t^3}}{2}\Big)\Big( 1-x(1+t+t^2+t^3) \Big)^{-1},
         \end{equation}
the conjecture says that for $2m \ge n$, $S'(m,n)$ is the number of paths
from $(0,0)$ to $( 4m-n-1, 2m-n+1)$, with up
$(\nearrow)$ steps $(1,1)$ and down $(\searrow)$ steps $(1,-1)$,
that avoid four consecutive up $(\nearrow)$ steps and never go
below the $x$-axis. There is no combinatorial meaning for $S'(m,n)$, where $ 2m < n$. Finally, we will find a redundant generating function for the numbers counting lattice paths with arbitrary given up steps, but one down step. In Theorem \ref{T:241} a redundant generating function is given that generalizes the previous two redundant generating functions (\ref{eq:200}) and (\ref{eq:199}).

\section{The ballot problem}
Let us consider the ballot problem \cite[pp.~1--8]{Mohanty} that
asks about the number $B(m,n)$ of lattice paths from $(1,0)$ to
$(m,n)$ (where $m>n$), with unit up $(\uparrow)$ steps $(0,1)$ and
unit right $(\rightarrow)$ steps $(1,0)$, that stay below the line
$x=y$. The number $B(m,n)$ can easily be computed by the
recurrence
    \begin{equation}\label{eq:201}
        B(m,n)=B(m-1,n)+B(m,n-1)\quad\text{for $0 \leq n < m$, $(m,n)\ne (1,0)$}
    \end{equation} with the
initial condition $B(1,0)=1$ and the boundary conditions
$B(m,-1)=0$ and $B(m,n)=0$ for $0\leq m \leq n$.
The recurrence $(\ref{eq:201})$ holds since a path ending at $(m,n)$ can be obtained from either a path ending at $(m-1,n)$ followed by a unit right
$(\rightarrow)$ step $(1,0)$ or a path ending at $(m,n-1)$
followed by a unit up $(\uparrow)$ step $(0,1)$.

Let $c(x)$ be the Catalan number generating function
\[ c(x)=\sum_{n=0}^\infty C_n x^n=\frac{1-\sqrt{1-4x}}{2x},\]
where $C_n=\frac1{n+1}\binom{2n}{n}$ is the $n$th Catalan number.

Now we want to find the generating function $\displaystyle \sum_{m \ge n \ge 0} B(m,n)x^m y^n$. From the well-known fact that $B(n+1,n)= C_n$,
we can use a variation of (\ref{eq:201}) which is
\begin{equation}\label{eq:202}
   B(m,n)  -  B(m-1,n) - B(m,n-1)  = \left\{
     \begin{array}{cl}
      1  & \mbox{if} \,  \,  (m,n)=(1,0) \\
      -  C_{m-1}  &  \mbox{if}  \,  \, m=n \\
        0 &  \mbox{otherwise}.  \,  \,
     \end{array}
   \right.
\end{equation}

Then the recurrence (\ref{eq:202}) is valid for all $m,n \ge 0$, and is also easy to see combinatorially.
Multiplying the recurrence (\ref{eq:202}) by $x^m y^n$ and summing on $m$ and $n$, we get
    \begin{equation*}
        (1-x-y) \sum_{m \ge n \ge 0 }B(m,n)x^m y^n=
        x - xyc(xy) .
    \end{equation*}
So, we have
\[ \sum_{m \ge n \ge 0}B(m,n)x^m y^n =\frac{x\Big(1-yc(xy) \Big)}{1-x-y}. \]
This is equivalent to
    \begin{equation}\label{eq:204}
        1+ \sum_{m \ge n \ge 0 }B(m,n)x^m y^n= \frac{1}{1-xc(xy)}.
    \end{equation}
Identity (\ref{eq:204}) can also be proved combinatorially by using the prime decomposition for Dyck paths \cite[pp.~1027--1030]{HC2}.

{ \begin{table}
 \centering
\begin{tabular}{c||rrrrrrr}
 $6$ & $-1$ & $-5$ & $-14$ & $-28$ & $-42$ & $-42$ & 0  \\
 $5$ & $-1$ & $-4$ & $-9$ & $-14$ & $-14$ & $0$ & $42$  \\
 $4$ & $-1$ & $-3$ & $-5$ & $-5$ & $0$ & $14$ & $42$  \\
 $3$ & $-1$ & $-2$ & $-2$ & $0$ & $5$ & $14$ & $28$  \\
 $2$ & $-1$ & $-1$ & $0$ & $2$ & $5$ & $9$ & $14$  \\
 $1$ & $-1$ & $0$ & $1$ & $2$ & $3$ & $4$ & $5$  \\
  $0$ & $0$ & $1$ & $1$ & $1$ & $1$ & $1$ & $1$  \\\hline \hline
  $n/m$ & $0$ & $1$ & $2$ & $3$ & $4$ & $5$ & $6$ \\
  \end{tabular}
\caption{The values of $\tilde{B}(m,n)$}\label{T:201}
\end{table} }

Instead of taking $B(m,n)=0$ for $m<n$, we could try to define
$B(m,n)$ everywhere in the first quadrant so that the recurrence (\ref{eq:201})
is satisfied as much as possible. We can do this by rewriting the
recurrence $(\ref{eq:201})$ as \[B(m-1,n)=B(m,n)-B(m,n-1) \textrm{ for } m<n, \]
with $B(m,n)=0$ for all $m<0$. So let us define $\tilde{B}(m, n)$
to be $B(m, n)$ for $m \ge n \ge 0$, and by
\begin{equation}\label{eq:2003}
\tilde{B}(m-1,n)=
\tilde{B}(m,n)-\tilde{B}(m,n-1) \textrm{ for } m<n,
\end{equation} with
$\tilde{B}(m,n)=0$ for all $m<0$. Table $\ref{T:201}$ shows the
values of $\tilde{B}(m,n)$ for the first quadrant. It is easy to
see that $\tilde{B}(m,n)$ can be extended to the whole first
quadrant so that the recurrence is satisfied everywhere except
$(1,0)$ and $(0,n)$ for $n \in \mathbb{P}=\{1,2,3, \ldots \}$. In fact, the recurrence (\ref{eq:2003}) is satisfied everywhere except at $(1,0)$ and $(0,1)$. To see this, we define $B'(m,n)$ by $B'(1,0)=1$, $B'(0,1)=-1$, $B'(m,-1)=0$, and $B'(-1,m)=0$ for all $m
\in \mathbb{N}=\{ 0,1,2,\ldots\}$ with the recurrence for all $(m,n) \in \mathbb{N}^2$, $(m,n) \neq (1,0)$ or $(m,n) \neq (0,1)$,
\[ B'(m,n)=B'(m-1,n)+B'(m,n-1).\]
Then it is easy to show by induction that $B'(m,n)=-B'(n,m)$ for all $m, n \in \mathbb{N}$, so we have that $B'(m,m)=0$ for all $m \in \mathbb{N}$.
Therefore, the values of $\tilde{B}(m,n)$ and
$B'(m,n)$ for all $(m,n)\in \mathbb{N}^2$ coincide.

From this fact, we can
assume that the recurrence for $B'(m,n)$ for the first quadrant is
\begin{equation}\label{eq:205}
   B'(m,n) - B'(m-1,n) - B'(m,n-1) = \left\{
     \begin{array}{cl}
       1  & \mbox{if} \,  \,  (m,n)=(1,0) \\
      - 1   &  \mbox{if}  \,  \, (m,n)=(0,1) \\
        0 &  \mbox{otherwise},
     \end{array}
   \right.
\end{equation}
with the boundary conditions $B'(m,n)=0$
for $m<0$ or $n<0$.

In terms of generating functions, the recurrence (\ref{eq:205})
 is equivalent to the equation
       \begin{equation*}
        (1-x-y) \sum_{m,n \geq 0} B'(m,n) x^my^n= x-y,
         \end{equation*}
 which is the same as
       \begin{equation} \label{eq:206}
         \sum_{m,n \geq 0} B'(m,n) x^my^n={x-y\over 1-x-y}.
         \end{equation}
Expanding in powers in
$x$ and $y$ and equating coefficients of $x^m y^n$ in equation
(\ref{eq:206}), we have the ballot number formula: \[
B'(m,n)={m+n-1\choose m-1}-{m+n-1\choose m}={m-n\over
m+n}{m+n\choose m}, \] which satisfies the recurrence
(\ref{eq:201}) with the initial and boundary conditions. That is,
\[ B(m,n)= B'(m,n) \textrm{ for all } 0 \le n \le m.\]
Also, we can see that the redundant generating function $(x-y)/
(1-x-y)$ for $B'(m,n)$ is much simpler than the generating
function $ x\big(1-yc(xy)\big)/(1-x-y)$ for $B(m,n)$.

\section{Variations of the ballot problem}\label{S:203}
Now let us consider \cite{Gessel-Ree} the line $x=py$, where $p \in
\mathbb{P}$, as the boundary line in the ballot problem instead of
the boundary line $x=y$. Then this problem asks about the number
$C_p(m,n)$ of lattice paths from $(1,0)$ to $(m,n)$, (where $m>p
n$) with unit up $(\uparrow)$ steps $(0,1)$ and unit right
$(\rightarrow)$ steps $(1,0)$, that never touch the line $x=py$.
Table \ref{T:202} shows the values of $C_2(m,n)$ in the case
$p=2$.
{ \begin{table}
 \centering
\begin{tabular}{c||rrrrrrrrr}
 $4$ &  &  &  &  &  &  &  &  & $0$ \\
 $3$ &  &  &  &  &  &  & $0$  & $12$ & $30$ \\
 $2$ &  &  &  &  & $0$ & $3$ & $7$  & $12$ & $18$\\
 $1$ &  &  & $0$ & $1$ & $2$ & $3$ & $4$ & $5$ & $6$\\
 $0$ & $0$ & $1$ & $1$ & $1$ & $1$ & $1$ & $1$ & $1$ & $1$ \\\hline \hline
$n/m$ & $0$ & $1$ & $2$ & $3$ & $4$ & $5$ & $6$ & $7$ & $8$ \\
 \end{tabular}
\caption{The values of $C_2(m,n)$}\label{T:202}
\end{table} }

If we extend $C_p(m,n)$ to the first quadrant then the recurrence seems to be satisfied everywhere except at $(1,0)$ and $(0,1)$. More precisely, let us define
values of $C'_p(m,n)$ by the recurrence
\[C'_p(m,n)=C'_p(m-1,n)+C'_p(m,n-1) ,\] and the initial conditions $C'_p(1,0)=1$, $C'_p(0,1)=-p$, and the boundary conditions $C'_p(m,-1)=0$ and $C'_p(-1,m)=0$ for all $m \in \mathbb{N}$.
This is equivalent to
\begin{equation}\label{eq:207}
   C'_p(m,n) - C'_p(m-1,n) - C'_p(m,n-1) = \left\{
     \begin{array}{cl}
       1  & \mbox{if} \,  \,  (m,n)=(1,0) \\
      - p   &  \mbox{if}  \,  \, (m,n)=(0,1) \\
        0 &  \mbox{otherwise} ,
     \end{array}
   \right.
   \end{equation}
with the boundary conditions $C'_p(m,n)=0$ for $m<0$ or $n<0$. The recurrence (\ref{eq:207}) is valid for all $m,n \ge 0$.

To show that $C'_p(m,n)=C_p(m,n)$ for $m \ge pn$, it is enough to show that $C'_p(pn,n)=0$ for all $n \in \mathbb{N}$. Multiplying the recurrence (\ref{eq:207}) by $x^m y^n$ and summing on $m$ and $n$, we get
       \begin{equation} \label{eq:208}
         \sum_{m,n \geq 0} C'_p(m,n) x^my^n
         ={ x-py \over 1-x-y}.
         \end{equation}
         Expanding in powers of $x$ and
$y$ and equating coefficients of $x^m y^n$ in equation
(\ref{eq:208}), we have the
 formula: for all $m,n \in \mathbb{N}$
\[ C'_p(m,n)={m+n-1\choose m-1}-p{m+n-1\choose
m}={m-p n\over m+n}{m+n\choose m}. \]
From the formula for $C'_p(m,n)$, it is easy to see that $C'_p(pn,n)=0$. Therefore the values of $C'_p(m,n)$ are the same as the values of $C_p(m,n)=0$ since $C'_p(pn,n)=0$ and $C_p(pn,n)=0$ for all $n \in \mathbb{N}$.

At this point it might be natural to ask what happens if we use
the line $y=px$ (where $p \in \mathbb{P}$) as the boundary line
instead of the line $y=x$. For convenience' sake we allow a
lattice path to touch the line $y=px$. Let $D_p(m,n)$ be the
number of lattice paths from $(0,0)$ to $(m,n)$ (where $n \le
pm$), with unit up $(\uparrow)$ steps $(0,1)$ and unit right
$(\rightarrow)$ steps $(1,0)$, that never cross the line $y=px$.
Let us consider the case $p=2$. Since the slope of the boundary line is $2$, we require $D_2(m,2m+1)=0$ and $D_2(m,2m+2)=0$ for all $m \in \mathbb{N}$. Then the number $D_2(m,n)$ can easily be computed by the recurrence
    \begin{equation*}
        D_2(m,n)=D_2(m-1,n)+D_2(m,n-1)\quad\text{for $ 0 \leq n \leq 2m $, $(m,n)\ne (0,0)$,}
    \end{equation*} with the initial condition $D_2(0,0)=1$ and the boundary conditions
    \begin{equation} \label{eq:240}
    D_2(m,2m+1)=0, D_2(m,2m+2)=0, \textrm{ and } D_2(m,-1)=0 \textrm{ for all } m \in \mathbb{N}.
     \end{equation}

{ \begin{table}
\centering
\begin{tabular}{c||rrrrrrrr}
5 &  &  & 0 & 12 & 43 &108 & 228 & 431 \\
4 &  & 0 & 3 &12 & 31 & 65 & 120 & 203 \\
3 &  & 0 & 3 & 9 & 19 & 34 & 55 & 83 \\
2 & 0 &1&3&6&10&15&21&28 \\
1 & 0 &1&2&3&4&5&6&7\\
0 & 1&1&1&1&1&1&1&1 \\\hline \hline
$n/m$ & 0 & 1 & 2 & 3 & 4 & 5 & 6 & 7 \\
\end{tabular}
\caption{The values of $D_2(m,n)$ for $ 0 \leq n \leq 2m
$}\label{T:203}
\end{table} }

Table \ref{T:203} shows the values of $D_2(m,n)$. From the table we
can observe that the numbers on the line $y=2x$ are the same as
the numbers (\seqnum{A001764}) on the line $x-1=2y$ in Table
\ref{T:202}, but the other numbers off the line $y=2x$ are different
from the numbers in Table \ref{T:202}.
We can prove that the numbers on the line $y=2x$ in Table \ref{T:203} are the same as
the numbers on the line $x-1=2y$ in Table
\ref{T:202} by finding a bijection between the set of lattice paths from $(0,0)$ to
$(pn,n)$, with unit up $(\uparrow)$ steps $(0,1)$ and unit right
$(\rightarrow)$ steps $(1,0)$, that never cross the line $x=py$
and the set of lattice paths from $(0,0)$ to $(n,pn)$, with unit
up $(\uparrow)$ steps $(0,1)$ and unit right $(\rightarrow)$ steps
$(1,0)$, that never cross the line $y=px$.

Now let us present the bijection. When we are given a
lattice path $P$ from $(0,0)$ to $(pn,n)$, first reverse it to
obtain a path $P^r$ from $(pn,n)$ to $(0,0)$ consisting of down
$(\downarrow)$ steps and left $(\leftarrow)$ steps and change a
down $(\downarrow)$ step with a right $(\rightarrow)$ step and a
left $(\leftarrow)$ step with a up $(\uparrow)$ step. Then we
have a lattice path $P'$ from $(0,0)$ to $(n,pn)$. In the same
way, we can transform the lattice path $P'$ from $(0,0)$ to
$(n,pn)$ into the lattice path $P$ from $(0,0)$ to $(pn,n)$. Since
the path $P$ does not cross the line $y=px$, we know that for each $i$,
 the number of left steps among the first $i$ steps in the path $P^r$ is always
less than or equal to $p$ times the number of down steps among
the first $i$ steps. So the corresponding path $P'$
to the path $P$ is a path that never crosses the line $y=px$,
and vice versa. Therefore we have
\[D_p(n,pn)=C_p(pn+1,n) \textrm{ for } n \ge 0 \textrm{ and }p \ge 1 .\]

From the path transformation, we can also derive
the number of lattice paths starting at a point that is below the
line $y=px$ and ending on the line $y=px$, with unit up
$(\uparrow)$ steps $(0,1)$ and unit right $(\rightarrow)$ steps
$(1,0)$, that never cross the line $y=px$. That is, the number
of lattice paths from $(i,j)$ to $(n,pn)$ (where $j \leq pi$)
having $n-i$ unit right $(\rightarrow)$ steps and $pn-j$ unit up
$(\uparrow)$ steps, that never cross the line $y=px$ is
$C_p\big(pn-j+1,n-i\big)$.

\subsection{Another way to get Catalan numbers}
If we compute, as before, the values of $D_2(m,n)$ using the
recurrence
\begin{equation} \label{eq:241}
D_2(m-1,n)=D_2(m,n)-D_2(m,n-1)
\end{equation} with
the boundary conditions $(\ref{eq:240})$ extended to the first quadrant, then we do not find anything interesting. So let us extend these values of $D_2(m,n)$
satisfying the recurrence $(\ref{eq:241})$ to the region $\{ \,(x,y) \in
\mathbb{Z}^2 \mid y \geq -x,y\geq 0\,\}$. There is no combinatorial significance for $D_2(m, n)$ where $n > 2m$.

Table \ref{T:204} shows the values of $D_2(m,n)$. Note that the numbers on the line $x=-2$ in Table \ref{T:204} are $(-1)^{n}M_n$, where $M_n$ is the $n$th Motzkin number
(\seqnum{A001006}). Surprisingly, we find that the number
$D_2(-n-1,n+1)$ is equal to $(-1)^{n+1} C_{n}$ for $n \in
\mathbb{N}$ on the boundary line $y=-x$ in the table.

{ \begin{table}
 \centering
\begin{tabular}{c||rrrrrrrrrrr}
 $6$ & $42$ & $28$ & $19$ & $13$ & $9$ & $6$ & $3$ & $0$ & $0$ & $12$ & $55$ \\
 $5$ & & $-14$ & $-9$ & $-6$ & $-4$ & $-3$ & $-3$ & $-3$ & $0$ & $12$ & $43$ \\
 $4$ & &  & $5$ & $3$ & $2$ & $1$ & $0$ & $0$ & $3$ & $12$ & $31$ \\
 $3$ & &  &  & $-2$ & $-1$ & $-1$ & $-1$ & $0$ & $3$ & $9$ & $19$ \\
 $2$ & &  &  &  & $1$ & $0$ & $0$ & $1$ & $3$ & $6$ & $10$ \\
 $1$ & &  &  &  &  & $-1$ & $0$ & $1$ & $2$ & $3$ & $4$ \\
 $0$ & &  &  &  &  &  & $1$ & $1$ & $1$ & $1$ & $1$ \\\hline \hline
 $n/m$ & $-6$ & $-5$ & $-4$ & $-3$ & $-2$ & $-1$ & $0$ & $1$ & $2$ & $3$ & $4$ \\
  \end{tabular}
\caption{The values of $D_2(m,n)$}\label{T:204}
\end{table} }

Before we find the redundant generating function for
$D_2(m,n)$ (We will give it after Theorem \ref{T:243}) let us prove
that we have the Catalan numbers on the boundary line $y=-x$ in
Table \ref{T:204}. Since $D_2(n,2n)=C_2(2n+1,n)$ for all $n \in \mathbb{N}$, we have
\[D_2(n,2n)={1 \over 2n+1}\binom {3n}n.\]
So, by the recurrence $(\ref{eq:241})$ with $D_2(n,2n+1)=0$, we have
\[D_2(n-1,2n+1) =- {1 \over (2n+1)}\binom {3n}n \textrm{ for all }n \in \mathbb{N} .\]

For each $n \in \mathbb{N}$, let us define $a_n$ to be $D_2(n-1,2n+1)$ and
$b_n$ to be $D_2(-n-1,n+1)$. Now let us find a
formula for $b_n$ in terms of $a_i$ in Table \ref{T:204}. As shown
in Figure \ref{F:201}, we can see \[D_2(m-1,n)=D_2(m,n)-D_2(m,n-1).\] By
iterating this recurrence, we can easily derive that
$b_n=\sum_{i=0}^\infty (-1)^{n-2i}c_{ni}a_i$, where $c_{ni}$ is
the number of paths from $(i-1,2i+1)$ to $(-n-1,n+1)$ with steps
$(-1,0)$ and $(-1,1)$. That is,
\[c_{ni}= \binom
{3i+n-2i}{3i}=\binom {n+i}{3i}.\]

The formula $b_{n} = (-1)^{n+1} C_{n}$ is a consequence of the following lemma. \vspace{.2in}

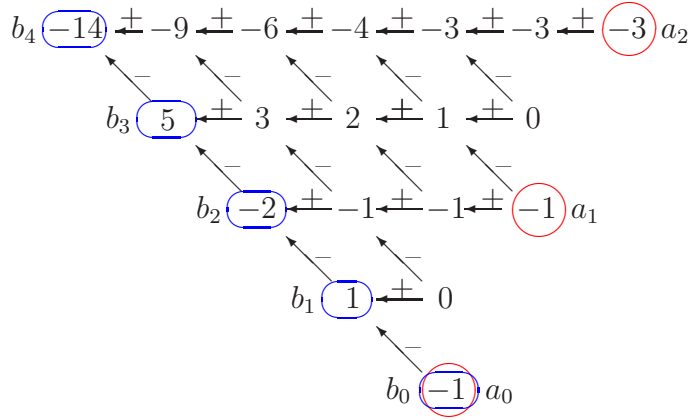
\begin{figure}[h]
\begin{center}
\setlength{\unitlength}{6mm}
\begin{picture}(12,8)
\put(-.4,8){$-14$}\put(1.85,8){$-9$}\put(3.85,8){$-6$}\put(5.85,8){$-4$}
\put(7.85,8){$-3$}\put(9.85,8){$-3$}\put(12,8){$-3$}
\put(2.1,5.95){$5$}\put(4.2,6){$3$}\put(6.2,6){$2$}
\put(8.2,6){$1$}\put(10.2,6){$0$} \put(3.8,4){$-2$}\put(6,4){$-1$}
\put(8,4){$-1$}\put(10,4){$-1$} \put(6.2,2){$1$}
\put(8.25,2){$0$}\put(8,0){$-1$}
\put(11.7,8.14){\vector(-1,0){.8}}
\put(9.7,8.14){\vector(-1,0){.8}}\put(9.9,6.6){\vector(-1,1){1}}
\put(7.7,8.14){\vector(-1,0){.8}}\put(7.9,6.6){\vector(-1,1){1}}
\put(5.7,8.14){\vector(-1,0){.8}}\put(5.9,6.6){\vector(-1,1){1}}
\put(3.7,8.14){\vector(-1,0){.8}}\put(3.9,6.6){\vector(-1,1){1}}
\put(1.7,8.14){\vector(-1,0){.6}}\put(1.9,6.6){\vector(-1,1){1}}
\put(9.9,6.2){\vector(-1,0){1}}\put(9.9,4.6){\vector(-1,1){1}}
\put(7.9,6.2){\vector(-1,0){1}}\put(7.9,4.6){\vector(-1,1){1}}
\put(5.9,6.2){\vector(-1,0){1}}\put(5.9,4.6){\vector(-1,1){1}}
\put(3.9,6.2){\vector(-1,0){1}}\put(3.9,4.6){\vector(-1,1){1}}
\put(9.65,4.2){\vector(-1,0){.8}}\put(5.2,6.25){+}
\put(7.9,4.2){\vector(-1,0){1}}\put(7.9,2.6){\vector(-1,1){1}}
\put(5.9,4.2){\vector(-1,0){1}}\put(5.9,2.6){\vector(-1,1){1}}
\put(7.9,2.2){\vector(-1,0){1}}\put(7.9,0.6){\vector(-1,1){1}}
\put(7.2,2.25){+}\put(7.2,4.25){+}\put(7.2,6.25){+}
\put(5.2,4.25){+}\put(3.2,6.25){+}\put(1.2,8.25){+}
\put(3.2,8.25){+}\put(5.2,8.25){+}\put(7.2,8.25){+}
\put(9.2,8.25){+}\put(11.2,8.25){+}
\put(9.2,6.25){+}\put(9.2,4.25){+}\put(7.2,6.25){+}
\put(7.54,.95){--}\put(7.54,2.95){--}\put(7.54,4.95){--}\put(7.54,6.95){--}
\put(5.54,2.95){--}\put(5.54,4.95){--}\put(5.54,6.95){--}
\put(3.54,4.95){--}\put(3.54,6.95){--}\put(1.54,6.95){--}
\put(9.54,4.95){--}\put(9.54,6.95){--}
\put(12.5,8.19){\textcolor{red}{\circle{1.1}}}\put(10.5,4.19){\textcolor{red}{\circle{1.1}}}
\put(8.5,.19){\textcolor{red}{\circle{1.1}}}
\put(0.2,8.19){\textcolor{blue}{\oval(1.4,.8)}}
\put(2.24,6.19){\textcolor{blue}{\oval(1.3,.8)}}
\put(4.24,4.19){\textcolor{blue}{\oval(1.3,.8)}}
\put(6.24,2.19){\textcolor{blue}{\oval(1.1,.8)}}
\put(8.5,0.19){\textcolor{blue}{\oval(1.3,.8)}}
\put(9.3,0.0){$a_0$}\put(7.1,0.0){$b_0$}\put(5,2.0){$b_1$}\put(2.9,4){$b_2$}
\put(-1.2,8){$b_4$}\put(.9,6){$b_3$}\put(13.2,8){$a_2$}\put(11.2,4){$a_1$}
\end{picture}
\end{center}
\caption{The values of $D_2(m,n)$ }\label{F:201}
\end{figure}

\begin{lemma}\label{le:02} For $n \ge 0$,
$$\sum_{i=0}^\infty \binom {n+i}{3i}
\frac{1}{2i+1}\binom{3i}{i} =\frac{1}{n+1}\binom{2n}{n}.$$
\end{lemma}

\begin{proof} Since the binomial coefficient $\binom{n+i}{3i}$ is nonzero only for $0 \leq i \leq \left\lfloor n/2 \right\rfloor$, the summation $i$ on the left is in that range. So, for $n \ge 0$
{\allowdisplaybreaks
\begin{align*}
\sum_{i=0}^\infty \binom {n+i}{3i}
\frac{1}{2i+1}\binom{3i}{i} &=\sum_{i=0}^{\left\lfloor n/2 \right\rfloor}
\frac{(n+i)!}{(3i)!\, (n-2i)!} \cdot \frac{1}{2i+1} \cdot \frac{(3i)!}{i!\, (2i)!}\\
&=\sum_{i=0}^{\left\lfloor n/2 \right\rfloor} \frac{(n+i)!}{(n-2i)! \, i!\, (2i+1)!}\\
&= \sum_{i=0}^{\left\lfloor n/2 \right\rfloor} \frac{(n+i)!}{(n-2i)! \, i!}\cdot \frac{n!}{n!\, (2i+1)!}\\
&= \frac{1}{n+1}\sum_{i=0}^\infty \binom{n+i}{n}
\binom{n+1}{n-2i}\\
&=\frac{1}{n+1}\binom{2n}{n},
\end{align*}
}
where the second last equality is only valid for $0 \leq i \leq \left\lfloor n/2 \right\rfloor$ and the last equality is derived from the fact that the coefficient of $x^n$ in $\big((1+x)/(1-x^2)\big)^{n+1}$ is the
same as the coefficient of $x^n$ in $1/(1-x)^{n+1}$.
\end{proof}
Note that in the paper \cite{Sun} Sun proved Lemma \ref{le:02} combinatorially by using binary and ternary trees. Another proof of Lemma \ref{le:02} is in the paper \cite{MS}.

So, we have the following result.
%
\begin{theorem}\label{T:221}
For each $n \in \mathbb{N}$, the number $D_2(-n-1,n+1)$ is equal to $(-1)^{n+1} C_{n}$, where $C_n$ is the $n$th
Catalan number.
\end{theorem}

Note that there is no combinatorial significance for $D_2(-n-1,n+1)$ for $n \ge 0$. We will give another proof of Theorem \ref{T:221} later.

\section{A conjecture of Niederhausen and Sullivan}
A generalized Dyck path is a path starting at $(0,0)$, with up $(\nearrow)$
steps $(1,1)$ and down $(\searrow)$ steps $(1,-1)$, that never goes below the $x$-axis. Now let us count the number of generalized Dyck paths without four
consecutive up $(\nearrow)$ steps. Define $S(m,n)$ to be the
number of such generalized Dyck paths from $(0,0)$ to $(m,n)$ that end with a down $(\searrow)$ step $(1,-1)$.
Since a path starts at $(0,0)$, it is natural to set the initial
condition $S(0,0)=1$. It is obvious that $S(2,0)=1,
S(3,1)=1, \textrm{ and } S(4,2)=1$. But we have that $S(5,3)=0$
because we don't allow four consecutive up steps. Since a path cannot have four consecutive up $(\nearrow)$ steps, such a generalized Dyck path ending at $(m,n)$ can be obtained uniquely from a generalized Dyck path ending at $(m-i-1,n-i+1)$ followed by $i$ consecutive up $(\nearrow)$ steps for $0\le i \le 3$ and a down
$(\searrow)$ step. So, when we define $S(m,n)=0$ for $n<0$ we have the recurrence relation
   \begin{equation*}
    S(m,n)=\sum_{i=-1}^2 S(m-i-2,n-i) \textrm{ for } m,n \in \mathbb{N}, (m,n)\neq(0,0).
   \end{equation*}
Table \ref{T:206} shows the values of $S(m,n)$ with the $0$'s omitted.

{ \begin{table}
 \centering
\begin{tabular}{c||rrrrrrrrrrrrrr}
  $6$ &  &  &   &   &   &   &   &   &   &  &  &  & $1$ &      \\
  $5$ &  &  &   &   &   &   &   &   &   &  &  & $3$ &  &  $19$  \\
  $4$ &  &  &   &   &   &   &   &   & $1$  &  & $6$ &  & $28$ &      \\
  $3$ &  &  &   &   &   &   &   & $2$ &   & $9$ &  & $33$ &  & $116$     \\
  $2$ &  &  &   &   & $1$ &   & $3$ &  & $10$  &    & $32$ &  & $101$ &    \\
  $1$ &  &  &   & $1$ &   & $3$ &   & $8$ &    & $23$ &  & $68$  &  & $205$ \\
  $0$ & $1$ &   & $1$ &   & $2$ &   & $5$ &   & $13$ &  & $36$ &   & $104$  \\\hline  \hline
  $n/m$  & $0$ & $1$ & $2$ & $3$ & $4$ & $5$ & $6$ & $7$ & $8$ & $9$ & $10$ & $11$ & $12$ & $13$ \\
  \end{tabular}
\caption{The values of $S(m,n)$ }\label{T:206}
\end{table} }

{ \begin{table}
 \centering
\begin{tabular}{c||rrrrrrr}
  $6$  & $2$ & $1$ & $0$ & $5$ & $32$ & $112$ & $297$ \\
  $5$  & $-1$ & $-1$ & $0$ & $8$ & $33$ & $90$ & $200$ \\
  $4$  & $1$ & $0$ & $2$ & $10$ & $28$ & $61$ & $115$  \\
  $3$  & $-1$ & $0$ & $3$ & $9$ & $19$ & $34$ & $55$ \\
  $2$  & $0$ & $1$ & $3$ & $6$ & $10$ & $15$ & $21$  \\
  $1$  & $0$ & $1$ & $2$ & $3$ & $4$ & $5$ & $6$ \\
  $0$  & $1$ & $1$ & $1$ & $1$ & $1$ & $1$ & $1$  \\
  \hline  \hline
  $n/m$  & $0$ & $1$ & $2$ & $3$ & $4$ & $5$ & $6$   \\
  \end{tabular}
\caption{The values of $S'(m,n)$ }\label{T:207}
\end{table} }

In the paper \cite{NS}, Niederhausen and Sullivan conjectured that the number $S'(m,n)$ for $m,n \in \mathbb{N}$ defined by
   \begin{equation*}
    \sum_{m, n = 0}^\infty  S'(m, n)x^{ m } t^{ n }
    =\Big( \frac{3+t-\sqrt{(1+t)^2+4t^3}}{2}\Big)\Big( 1-x(1+t+t^2+t^3) \Big)^{-1}
         \end{equation*}
is equal to $S(4m-n,2m-n)$ in the case $2m \ge n$. That is, in the case $2m \ge n \ge 0$, $S'(m,n)$ is the number of generalized Dyck paths from $(0,0)$ to $(4m-n,2m-n)$ without four consecutive up steps that end with a down $(\searrow)$ step $(1,-1)$. Note that there is no combinatorial meaning for $S'(m,n)$ when $2m < n$.
Table \ref{T:207} shows the values of $S'(m,n)$.

\section{Main Theorem}
Now we will prove a generalization of the Niederhausen and Sullivan conjecture involving more general paths. To do this, let us define a step set $T$ to be a subset of the set $\mathbb{N}\cup \{ -1\}$. We assume that the step set $T$ has $-1$ as an element and that $K$ is the largest element of $T$. Now let us consider a path from $(0,0)$ to
$(m,0)$ (where $m \in \mathbb{N}$), with steps $(1,i)$ (where $i
\in T $), that never goes below the $x$-axis. Put a weight of $c_i
x$ on each step $(1,i)$, where $c_i$ is an arbitrary weight for
all $i \in T $ except $c_{-1}=1$ and $c_{K}=1$, and let $f(x)$ be
the generating function for weighted paths from $(0,0)$ to
$(m,0)$, with steps $(1,i)$ (where $i \in T $), that never go
below the $x$-axis.

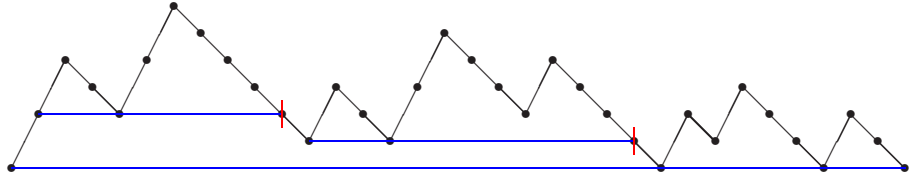
\begin{figure}[h]
\begin{center}
\setlength{\unitlength}{3.6mm}
\begin{picture}(33,6)
\put(0,0){\line(1,2){2}}\put(2,4){\line(1,-1){2}}\put(4,2){\line(1,2){2}}\put(6,6){\line(1,-1){5}}
\put(11,1){\line(1,2){1}}\put(12,3){\line(1,-1){2}}\put(14,1){\line(1,2){2}}
\put(16,5){\line(1,-1){3}}\put(19,2){\line(1,2){1}}
\put(20,4){\line(1,-1){4}}\put(24,0){\line(1,2){1}}
\put(25,2){\line(1,-1){1}}\put(26,1){\line(1,2){1}}
\put(27,3){\line(1,-1){3}}\put(30,0){\line(1,2){1}}
\put(31,2){\line(1,-1){2}}\put(25,2){\line(1,-1){1}}
\put(0,0){\circle*{0.3}}\put(1,2){\circle*{0.3}}
\put(2,4){\circle*{0.3}}\put(3,3){\circle*{0.3}}
\put(4,2){\circle*{0.3}}\put(5,4){\circle*{0.3}}
\put(6,6){\circle*{0.3}}\put(7,5){\circle*{0.3}}
\put(8,4){\circle*{0.3}}\put(9,3){\circle*{0.3}}
\put(10,2){\circle*{0.3}}\put(11,1){\circle*{0.3}}
\put(12,3){\circle*{0.3}}\put(13,2){\circle*{0.3}}
\put(14,1){\circle*{0.3}}\put(15,3){\circle*{0.3}}
\put(16,5){\circle*{0.3}}\put(17,4){\circle*{0.3}}
\put(18,3){\circle*{0.3}}\put(19,2){\circle*{0.3}}
\put(20,4){\circle*{0.3}}\put(21,3){\circle*{0.3}}
\put(22,2){\circle*{0.3}}\put(23,1){\circle*{0.3}}
\put(24,0){\circle*{0.3}}\put(25,2){\circle*{0.3}}
\put(26,1){\circle*{0.3}}\put(27,3){\circle*{0.3}}
\put(28,2){\circle*{0.3}}\put(29,1){\circle*{0.3}}
\put(30,0){\circle*{0.3}}\put(31,2){\circle*{0.3}}
\put(32,1){\circle*{0.3}}\put(33,0){\circle*{0.3}}
\put(1,2){\textcolor{blue}{\line(1,0){9}}}
\put(11,1){\textcolor{blue}{\line(1,0){12}}}
\put(0,0){\textcolor{blue}{\line(1,0){33}}}
\put(10,1.5){\textcolor{red}{\line(0,1){1}}}
\put(23,.5){\textcolor{red}{\line(0,1){1}}}
\end{picture}
\end{center}
\caption{ Decomposition for a path from $(0,0)$ to $(m,0)$ when
$T =\{-1,2\}$ }\label{F:202}
\end{figure}

As shown in Figure \ref{F:202}, any such path starting
with a step $(1,i)$, where $i \ge 0$, can be decomposed uniquely
\cite[pp.~1027--1030]{HC2} by cutting it right before the first
points at height $j$ for $j=i-1, i-2, \ldots ,0$. This gives the
functional equation for $f$:
        \begin{equation} \label{eq:214}
         f(x)= \sum_{ i \in T }c_i \big(xf(x)\big)^{i+1},
         \end{equation} where the term for $i=-1$ corresponds to the empty path.

Now let us consider a path $Q$ from $(0,h)$ to
$(m,0)$ (where $h, m \in \mathbb{N}$), with steps $(1,i)$ (where $i
\in T $), that never goes below the $x$-axis. Decomposing the
path $Q$, as before, into the first time it reaches height $i$ for
$i=h-1,h-2, \ldots ,0$, we can deduce that the
generating function for such weighted paths from $(0,h)$ to
$(m,0)$ is $x^hf(x)^{h+1}$.

For $m, n \ge 0$ and $h \ge 0$, we define $P_h(m,n)$ to be the sum of the weights of lattice paths from $(0,h)$ to $(m,n)$, with steps $(1,i)$ where $i \in T $,
that never go below the $x$-axis. Then we have the initial condition
$P_h(0,h)=1$ and $P_h(m,n)=0$ for $n <0 $ because paths cannot go below the $x$-axis. Since such paths have steps $(1,i)$ (where $i \in
         T$) we have the recurrence
        \begin{equation} \label{eq:215}
         P_h(m+1,n)=  \sum_{ i \in T } c_i P_h(m,n-i) \textrm{ for } m,n \ge 0.
         \end{equation}
Then we can easily deduce
\[x^hf(x)^{h+1} =\sum_{m=0}^\infty P_h(m,0)x^m \text{ and } f(x)=
\sum_{m=0}^\infty P_0(m,0)x^m. \]
Also, we find the generating function for $P_h(m,n)$:
\begin{equation} \label{eq:248}
 \sum_{m,n=0}^\infty P_h(m,n) x^m t^n=\frac{t^h}{1-\sum_{i \in T}(c_i t^i x)}-\frac{\big(xf(x)\big)^{h+1} t^{-1}}{1-\sum_{i \in T}(c_i t^ix)},
 \end{equation}
where the first term on the right side in $(\ref{eq:248})$ counts the sum of the weights of all lattice paths start at $(0,h)$, and the second term counts the sum of the weights of all lattice paths start at $(0,h)$ and go below the $x$-axis.

Now we want to extend the region of definition of $P_h(m,n)$ to all integer values of $n$. That is, for $m \ge 0$ and $n<0$, we want the recurrence (\ref{eq:215}) to
hold below the $x$-axis. So, for $m \ge 0$ let us define $P'_h(m,n)$ to be
$P_h(m,n)$ for $n\geq -K$, and by
        \begin{equation} \label{eq:216}
            P'_h(m,n)=P'_h(m+1,n+K) - \sum_{  i \in T - \{K\} }
            c_i P'_h(m,n+K-i) \quad \text{for $n<-K$}.
        \end{equation}
Note that $P'_h(m,n)$ in (\ref{eq:216}) holds for all $m \in \mathbb{N}$ and $n \in \mathbb{Z}$.
{ \begin{table}
 \centering
\begin{tabular}{c||rrrrrrrr}
  $4$  & $0$ & $0$ & $2$ & $4$ & $12$ & $32$ & $72$ & $194$ \\
  $3$  & $0$ & $1$ & $1$ & $6$ & $9$ & $27$ & $60$ & $137$ \\
  $2$  & $0$ & $1$ & $2$ & $3$ & $10$ & $16$ & $44$ & $93$ \\
  $1$  & $1$ & $0$ & $2$ & $2$ & $5$ & $12$ & $21$ & $56$  \\
  $0$  & $0$ & $1$ & $0$ & $2$ & $2$ & $5$ & $12$ & $21$ \\
  $-1$  & $0$ & $0$ & $0$ & $0$ & $0$ & $0$ & $0$ & $0$  \\
  $-2$  & $0$ & $0$ & $0$ & $0$ & $0$ & $0$ & $0$ & $0$  \\
  $-3$  & $0$ & $-1$ & $0$ & $-2$ & $-2$ & $-5$ & $-12$ & $-21$  \\
  $-4$  & $0$ & $1$ & $0$ & $2$ & $2$ & $5$ & $12$ & $21$ \\
  $-5$  & $-1$ & $-1$ & $-2$ & $-4$ & $-7$ & $-17$ & $-33$ & $-77$  \\
  $-6$  & $2$ & $2$ & $4$ & $8$ & $14$ & $34$ & $66$ & $154$ \\
  $-7$  & $-3$ & $-5$ & $-8$ & $-17$ & $-33$ & $-72$ & $-155$ & $-345$  \\
  $-8$  & $6$ & $10$ & $18$ & $35$ & $74$ & $155$ & $342$ & $762$ \\
  $-9$  & $-13$ & $-20$ & $-39$ & $-76$ & $-160$ & $-344$ & $-753$ & $-1696$  \\
  $-10$  & $26$ & $43$ & $82$ & $167$ & $348$ & $758$ & $1670$ & $3759$ \\
   \hline  \hline
  $n/m$  & $0$ & $1$ & $2$ & $3$ & $4$ & $5$ & $6$ & $7$   \\
  \end{tabular}
\caption{The values of $P'_1(m,n)$ with $T=\{-1, 1, 2\}$ }\label{T:209}
\end{table} }

Table \ref{T:209} shows the values of $P'_1(m,n)$ with $T=\{-1, 1, 2\}$ and $c_i=1$ for all $i \in T$. In this table once we know the values of $P'_1(0,n)$ in the first column, we can figure out the values of $P'_1(m,n)$ for $m>0$ by the recurrence that corresponds to the step set $T=\{-1, 1, 2\}$. For example, $P'_1(1,n)$ can be computed by summing $P'_1(0,n-i)$ for all $i \in T$. Similarly,  $P'_1(m+1,n)$ can be computed by summing $P'_1(m,n-i)$ for all $i \in T$ and $m \ge 0$. So, the values of $P'_h(m,n)$ are determined by the values of $P'_h(0,n)$. Therefore, it is important to find out the generating function for $P'_h(0,n)$ to know the values of $P'_1(m,n)$.

In the case $T =\{ -1, 0, 1, 2\}$ with $c_i=1$ for all $i \in T $, we will prove $P'_0(m,n)=S(2m+n,n)$ (where $m,n \in \mathbb{N}$) which is related to the Niederhausen and Sullivan conjecture. Now we present a direct way to see the relation between the numbers $S(m,n)$ and $P'_0(m,n)$ by giving a bijection between the set $A$ of generalized Dyck paths from $(0,0)$ to $(2m+n,n)$ (where $n \ge 0$) with no four consecutive up steps that end a down step and the set $B$ of paths from $(0,0)$ to $(m,n)$ with steps $(1,-1)$, $(1,0)$, $(1,1)$, and $(1,2)$ that never go below the $x$-axis. Then the number $P'_0(m,n)$ becomes the cardinality of the set $B$. Then the following lemma gives us the bijection between the sets $A$ and $B$.

\begin{lemma}
For $m,n \in \mathbb{N}$, the number $S(2m+n,n)$ is  equal to the number $P'_0(m,n)$.
\end{lemma}

\begin{proof} Let us be given a generalized Dyck path ending at $(2m+n,n)$ that ends a down $(\searrow)$ step $(1,-1)$. Then since an up step $(\nearrow)$ of the path is $(1,1)$ and a down step $(\searrow)$ is $(1,-1)$, we can compute that the path has $m+n$ up steps and $m$ down steps.

Then it is enough to prove this lemma by presenting a bijection from the set $A$ to the set $B$. Now let us decompose the path at every down step of the path. Then we have the following four kinds of subpaths: $\searrow$, $\nearrow \searrow$, $\nearrow \nearrow \searrow$, and $\nearrow \nearrow \nearrow \searrow$. Now let us convert each subpath $\searrow $ to a $(1,-1)$ step, $\nearrow \searrow $ to a $(1,0)$ step, $\nearrow \nearrow \searrow $ to a $(1,1)$ step, and $\nearrow \nearrow \nearrow \searrow $ to a $(1,2)$ step. Then we can get a path with steps $(1,-1)$, $(1,0)$, $(1,1)$, and $(1,2)$ that never goes below the $x$-axis.
Similarly, when we are given a path with steps $(1,-1)$, $(1,0)$, $(1,1)$, and $(1,2)$ that never goes below the $x$-axis, we can have a generalized Dyck path ending at height at least $0$ by converting each $(1,-1)$ step to $\searrow $, each $(1,0)$ step to $\nearrow \searrow$, each $(1,1)$ step to $\nearrow \nearrow \searrow$, and each $(1,2)$ step to $\nearrow \nearrow \nearrow \searrow$.
\end{proof}

From the recurrence (\ref{eq:216}) for $P'_h(m,n)$, we have the
following lemma.
\begin{lemma}\label{L:241}
Let $T \subset \mathbb{N} \cup \{ -1\}$ be a step set and let $K$ be the largest element of $T$. We assume that $-1 \in T$, $K \neq -1$, $c_K=1$, and $c_{-1}=1$. Then we have for $h, m \in \mathbb{N}$, $P'_h(m,-K-1)=-P'_h(m,0)$.
\end{lemma}

\begin{proof}
By the recurrence (\ref{eq:216}) we have
\begin{align*}
P'_h(m,-K-1)&=P'_h(m+1,-1) - \sum_{i \in T-\{K\} }c_i P'_h(m,-1-i)\\
&=P'_h(m+1,-1) -P'_h(m,0)\\
&= -P'_h(m,0),
\end{align*} where the second and third equalities hold because $P'_h(m,n)=0$ for $-K \le n \le -1$.
\end{proof}

Now we are trying to find the sum of the weights of the paths with a given step set $T$. For any $h, K \in \mathbb{N}$, let us define $\mathcal{A}(x)$ to be the generating function for $-P'_h(m,0)$:
{\allowdisplaybreaks
\begin{align}
\label{eq:217}
           \mathcal{A}(x)&=-\sum_{m=0}^\infty P'_h(m,0)x^m=-x^hf(x)^{h+1},
           \intertext{and, for $j \in \mathbb{Z}$, let $\mathcal{A}_j(x)$ be the generating function for $P'_h(m,-K-1-j)$:}
           \mathcal{A}_j(x)&=\sum_{m=0}^\infty P'_h(m,-K-1-j)x^m. \nonumber
        \end{align}
        }
        Then $\mathcal{A}_j(x)=0$ for $-K \le j \le -1$ since $P'_h(m,n)=0$ for $-K \le n \le -1$, and by Lemma \ref{L:241} we know $ \mathcal{A}_0(x)=\mathcal{A}(x)$ and $\mathcal{A}_{-K-1}(x)=- \mathcal{A}(x)$.

Multiplying both sides of the recurrence (\ref{eq:216}) by $x^m$ and summing on $m$
gives
        \begin{equation} \label{eq:218}
        \sum_{m=0}^\infty P'_h(m,n)x^m = \sum_{m=0}^\infty P'_h(m+1,n+K)x^m -
        \sum_{  i \in T - \{K\} }c_i \sum_{m=0}^\infty
        P'_h(m,n+K-i)x^m .
         \end{equation}
By setting $n=-K-1-j$, where $j \ge 0$, in identity (\ref{eq:218}) we
have
        \begin{equation} \label{eq:219}
         \mathcal{A}_j(x) = \frac{1}{x}(\mathcal{A}_{j-K}(x)-\mathcal{A}_{j-K}(0)) -\sum_{  i \in T - \{K\} }
         c_i \mathcal{A}_{i+j-K}(x) .
         \end{equation}
Let $\mathbb{A}(x,t) := \sum_{j=0}^\infty \mathcal{A}_j(x)t^j$. Multiplying both sides of identity (\ref{eq:219}) by $t^j$ and
summing on $j $ gives
       \begin{align} \label{eq:242}
         \mathbb{A}(x,t) = & \frac1{x}\biggl( \sum_{0 \le j \le K-1} \mathcal{A}_{j-K}(x)t^j + t^K
          \mathbb{A}(x,t) - \sum_{0 \le j \le K-1} \mathcal{A}_{j-K}(0)t^j \nonumber \\
          & - t^K\sum_{j=0}^\infty \mathcal{A}_j(0)t^j \biggl) -  \sum_{i \in T-\{K\}} c_i \biggl( \sum_{0 \le j \le K-i-1} \mathcal{A}_{i+j-K}(x)t^j + t^{K-i} \mathbb{A}(x,t) \biggl).
         \end{align}

The fifth term on the right side in (\ref{eq:242}) is $\mathcal{A}(x)$ from the following simplification:
       \begin{align*}
        - \sum_{i \in T-\{K\}} c_i \sum_{0 \le j \le K-i-1} \mathcal{A}_{i+j-K}(x)t^j  = &
         - \sum_{0 \le j \le K}  \mathcal{A}_{j-K-1}(x)t^j \\
         & - \sum_{i \in T-\{K,-1\}} c_i \sum_{0 \le j \le K-i-1} \mathcal{A}_{i+j-K}(x)t^j \\
         =  & - \sum_{0 \le j \le K} \mathcal{A}_{j-K-1}(x)t^j \\
         = &  \mathcal{A}(x),
                           \end{align*}
                           where the second and third equalities hold because $\mathcal{A}_j(x)=0$ for $-K \le j \le -1$ and $\mathcal{A}_{-K-1}(x)=- \mathcal{A}(x)$.

Therefore, collecting terms in (\ref{eq:242}) gives
\begin{equation} \label{eq:220}
         \mathbb{A}(x,t) \Big( 1- \frac{t^K}{x} + \sum_{  i \in T - \{K\} }c_i t^{K-i}
\Big) = \mathcal{A}(x)- \frac{ t^{K}}{x} \sum_{j=0}^\infty \mathcal{A}_j(0)t^j.
         \end{equation}

To find the generating function for $P'_h(0,n)$, we need to determine $\mathbb{A}(0,t)$. To do this, we work in the ring of Laurent series $\mathbb{C}((x))[[t]]$. We cannot set $x=0$ directly in (\ref{eq:220}) but we can use the following lemma to find the constant term in $x$. Let us define $\text{CT}_x$ to be the constant term of a power series in $x$.

\begin{lemma}\label{L:242}
Let $\mathcal{H}(x)$ be a power series in $x$. Let $\alpha$ and $\beta$ be power series in $t$ with no constant term. Then in $\mathbb{C}((x))[[t]]$ the constant term in $x$ of $\mathcal{H}(x) /( 1-\alpha x^{-1}-\beta )$ is $\mathcal{H}\big( \alpha /(1 - \beta)\big) / (1- \beta)$.
\end{lemma}
\begin{proof} Let $\mathcal{H}(x)=\sum_{n=0}^\infty a_n x^n$. Then
{\allowdisplaybreaks
\begin{align}
\text{CT}_x {\mathcal{H}(x) \over 1-\alpha x^{-1}-\beta }&=
\text{CT}_x {\mathcal{H}(x) \over (1-\beta)\big(1-\alpha x^{-1}/(1-\beta)\big)}  \nonumber\\
&= {1 \over 1-\beta} \, \text{CT}_x {\mathcal{H}(x) \over \big(1-\alpha x^{-1}/(1-\beta) \big)} \nonumber \\
&= {1\over 1- \beta} \, \text{CT}_x \mathcal{H}(x) \, \sum_{m=0}^\infty\Big(
{\alpha
\over 1- \beta} \Big)^m x^{-m} \nonumber \\
&= {1\over 1- \beta} \, \text{CT}_x \sum_{n=0}^\infty a_n x^n \,
\sum_{m=0}^\infty \Big( {\alpha
\over 1- \beta} \Big)^m x^{-m} \nonumber \\
&= {1\over 1- \beta}\, \mathcal{H}\Big( {\alpha \over 1- \beta} \Big).
\nonumber \qedhere
\end{align}
}
\end{proof}

Now let us find the constant term in $x$ in $\mathbb{A}(x,t)$ by Lemma \ref{L:242}. Dividing identity (\ref{eq:220}) by $1-
t^K/x + \sum_{ i \in T - \{K \} }c_i t^{K-i}$ gives
{\allowdisplaybreaks
\begin{align} \label{eq:221}
\displaystyle \mathbb{A}(x,t) = \frac{ \mathcal{A}(x)}{\displaystyle 1- \frac{t^K}{x} + \sum_{
i \in T - \{K \} }c_i t^{K-i} }  -
\frac{\displaystyle \frac{t^K}{x}  \sum_{j=0}^\infty \mathcal{A}_j(0)t^j}{
\displaystyle1- \frac{t^K}{x} + \sum_{ i \in T - \{K\} }c_i
t^{K-i} }.
\end{align}
}
Since the second term on the right-hand side of
(\ref{eq:221}) contains only negative powers of $x$, we get the constant term in $x$ in $\mathbb{A}(x,t)$ as follows:
\begin{align*}
\text{CT}_x \mathbb{A}(x,t) &= \text{CT}_x \frac{ \mathcal{A}(x)}{1-
\frac{t^K}{x} + \sum_{ i
\in T - \{K \} }c_i t^{K-i} } \nonumber \\
&=\frac{1}{1+ \sum_{ i \in T - \{K \} }c_i t^{K-i} }
 \, \mathcal{A} \Big( \frac{t^K}{ 1+\sum_{ i \in T - \{K \} }c_i t^{K-i} } \Big) \nonumber \\
&=- \frac{t^{hK}}{(1+ \sum_{ i \in T - \{K \} }c_i t^{K-i})^{h+1} }
 \, f \Big( \frac{t^K}{ 1+\sum_{ i \in T - \{K \} }c_i t^{K-i} } \Big)^{h+1}
 ,
\end{align*}
where the second equality follows from Lemma \ref{L:242} and the last equality follows from (\ref{eq:217}).

To simplify the constant term $\mathbb{A}(0,t)$ that we computed above, let us define a power series $g$ by
        \begin{equation} \label{eq:222}
         g(t)=  \frac{1}{\mathcal{B}(t)}f\Big(\frac{t^K}{\mathcal{B}(t)}\Big),  \text{ where }\mathcal{B}(t)=\sum_{ i \in T }c_i t^{K-i} .
         \end{equation}
Then we have
\begin{align*}
-t^{hK}g(t)^{h+1} &= \mathbb{A}(0,t) = \sum_{j=0}^\infty \mathcal{A}_j(0)t^j  \nonumber \\
&= \sum_{j=0}^\infty P'_h(0,-K-1-j)t^j .
\end{align*}
Since $P'_h(0,n)=0$ for $n \ge -K$ except for $P'_h(0,h)=1$, we have
\[ \sum_{n \in \mathbb{Z}} P'_h(0,h-n)t^n = \sum_{n \in \mathbb{N}} P'_h(0,h-n)t^n = 1-g(t)^{h+1}t^{(h+1)(K+1)}. \]

Finally, from the above generating function for $P'_h(0,h-n)$, we can find the generating function for $P'_h(m,n)$. To do this, we rewrite the recurrence $(\ref{eq:216})$ as
\[ P'_h(m+1,n+K)=\sum_{i \in T}c_i P'_h(m,n+K-i). \]
Then replacing $n$ with $Km+h-n$ gives
\[ P'_h \big(m+1,K(m+1)+h-n \big)=\sum_{i \in T}c_i P'_h \big(m,K(m+1)+h-n-i \big). \]
Next, multiplying both sides by $t^n$ and summing on $n$ gives
\[ \sum_{n=0}^\infty P'_h \big(m+1,K(m+1)+h-n\big)t^n= \mathcal{B}(t) \sum_{n=0}^\infty  P'_h(m,Km+h-n)t^n. \]
Finally, we conclude
{\allowdisplaybreaks
\begin{align*}
\sum_{n=0}^\infty P'_h(m,Km+h-n)t^n &= \mathcal{B}(t)^m \sum_{n=0}^\infty  P'_h(0,h-n)t^n \nonumber \\
 &= \mathcal{B}(t)^m \big( 1-g(t)^{h+1}t^{(h+1)(K+1)} \big). \nonumber \\
\end{align*}
}
We summarize with the following theorem.
\begin{theorem}\label{T:241} Let $T \subset \mathbb{N}\cup \{ -1\}$ be a step set and let $K$ be the largest element of $T$. We assume that $-1 \in T$ and $K \neq -1$. Let $P_h(m,n)$ be the sum of the weights of the paths from $(0, h)$ to $(m, n)$ (where $h, m, n \in \mathbb{N}$), with steps $(1, i)$ (where $i \in T$), that never go below the $x$-axis, where we put a weight of $c_i x$ on each step $(1,i)$ (where $c_i$ is an arbitrary weight for all $i \in T$ except $c_{-1} = 1$ and $c_K = 1$). For each $m \ge 0$, define $P'_h(m,n)$ for $n \le Km+h$ by
\begin{align*}
\mathcal{P}_h(x,t) = & \sum_{m,\, n =0}^\infty P'_h(m,Km+h-n)x^m t^n \nonumber \\
= & \Big(1- g(t)^{h+1} t^{(h+1)(K+1)} \Big) \Big(1-x \mathcal{B}(t)
         \Big)^{-1},
\end{align*}
where the power series $g(t)$ is uniquely determined by
\begin{equation} \label{eq:250} \sum_{ i \in T - \{0, -1 \}
} c_i t^{K-i}g(t) \sum_{n=1}^i \big(t^{K+1}g(t) \big)^{n-1}  =1,
\end{equation}
and $P'_h(m,n)=0$ for $n > Km+h$.
Then we have \[ P'_h(m,n)=P_h(m,n) \textrm{ for } m,n \ge 0.\]
\end{theorem}
\begin{proof} It is enough to show that the power series $g$ defined by (\ref{eq:222}) satisfies the functional equation (\ref{eq:250}) and $g$ is uniquely determined by (\ref{eq:250}). From the functional equation (\ref{eq:214}) for $f$, we can find
the functional equation for $g$. Substituting $t^K/ \mathcal{B}(t) $ for $x$ in equation (\ref{eq:214})
and simplifying gives the functional equation for $g$
        \begin{equation} \label{eq:223}
         \mathcal{B}(t)g(t) = \sum_{ i \in T }c_i \big( t^K g(t) \big)^{i+1}.
         \end{equation}
Since $\displaystyle \mathcal{B}(t)=\sum_{ i \in T}c_i t^{K-i} $
where $c_{-1}=1$ and $c_K=1$, our theorem follows from the factorization:

\begin{align*}
\biggl( & \sum_{ i \in T }c_i  t^{K-i} \biggl)g(t) - \sum_{ i
\in T}c_i \big( t^K g(t) \big)^{i+1} \\
& \qquad\qquad = \biggl( t^{K+1} + \sum_{ i \in T - \{0, -1 \} }c_i
t^{K-i} \biggl) g(t) -\biggl(
1+ \sum_{ i \in T - \{0, -1 \} }c_i \big( t^K g(t) \big)^{i+1} \biggl)\\
& \qquad\qquad =  t^{K+1}g(t) + \sum_{ i \in T - \{0, -1 \} }c_i
 t^{K-i}g(t)
-1 - \sum_{ i \in T - \{0, -1 \} }c_i \big( t^K g(t) \big)^{i+1} \\
& \qquad\qquad = \Big(  t^{K+1}g(t) -1 \Big) \biggl( 1 - \sum_{ i \in T - \{0, -1
\} } c_i  t^{K-i}g(t) \sum_{n=1}^i \big( t^{K+1} g(t) \big)^{n-1} \biggl).
\end{align*}

Now we want to show that the power series $g$ is uniquely determined by (\ref{eq:250}). Every term, except $i=K$ and $n=1$, on the left side of (\ref{eq:250}) has a factor $t^j$ for $j \ge 1$ and the term for $i=K$ and $n=1$ is $g(t)$. Then the functional equation (\ref{eq:250}) may be written as
\[g(t)=1+t \mathcal{T}(t,g(t)), \textrm{ for some polynomial $\mathcal{T}$ in $t$}. \]
Equating coefficients gives a unique power series solution. Therefore the power series $g$ is uniquely determined by (\ref{eq:250}).
\end{proof}
Note that in Theorem \ref{T:241} we prefer to take $P'_h(m,Km+h-n)$ instead of $P'_h(m,n)$ as the coefficient of $x^m t^n$ in $\mathcal{P}_h(x,t)$ because we want $\mathcal{P}_h(x,t)$ to be a power series, not a Laurent series.
Comparing the generating function (\ref{eq:248}) for $P_h(m,n)$ with the generating function for $P'_h(m,n)$ in Theorem \ref{T:241}, we can see that the degree of equation (\ref{eq:250}) that $g$ satisfies is one less than the degree of equation (\ref{eq:214}) that $f$ satisfies.
In particular, when the largest element of the step set $T$ is $2$ (that is, $K=2$) we know that the degree of $f$ is $3$ and the degree of the corresponding function $g$ is $2$. So, while the power series $g$ can be easily computed by using the quadratic formula, the power series $f$ cannot be. We will see more details with examples in the following subsection.

\subsection{Corollaries to our main theorem}
First let us consider the case $T=\{-1, 1\}$ and $h=0$. Then we will see that Theorem \ref{T:241} gives (\ref{eq:206}) in this case. Since we defined $B(m,n)$ to be the number of lattice paths from $(1,0)$ to $(m,n)$ (where $m>n$), with unit up steps $(0,1)$ and unit right steps $(1,0)$, that stay below the line $x=y$ we can conclude
\[B(m+1,n)=P'_0(m+n, m-n) \textrm{ for } m \ge n.\]
Since the functional equation for $g$ in Theorem \ref{T:241} is $g(t)=1$, we have
\begin{equation}\label{eq:249}
\sum_{m,n = 0}^\infty  P'_0(m, m-n)x^m t^n = \frac{1-t^2}{1-x(1+t^2)}.
\end{equation}
Replacing $t^2$ with $t$ and $n$ with $2n$ in $(\ref{eq:249})$ gives
\begin{equation}\label{eq:262}
 \sum_{m,n = 0}^\infty
  P'_0(m, m-2n) x^{m} t^n = \frac{1-t}{1-x(1+t)}.
\end{equation}
Next, replacing $m$ with $m+n$ and $t$ with $t/x$ in $(\ref{eq:262})$ gives
\begin{equation}\label{eq:265}
 \sum_{\substack{
               n \ge 0  \\
             m \ge -n }}
  P'_0(m+n, m-n) x^{m} t^n = \frac{x-t}{x(1-x-t)}.
\end{equation}
Since the right side of equation $(\ref{eq:265})$ does not have terms $x^i t^j$ for $i < -1$, we reduce the range from $m \ge -n$ to $m \ge -1$ on the left side.
That is, \begin{equation}\label{eq:267}
 \sum_{\substack{
               n \ge 0  \\
             m \ge -1 }}
  P'_0(m+n, m-n) x^{m} t^n = \frac{x-t}{x(1-x-t)}.
\end{equation}
Finally, multiplying by $x$ and substituting $m+1$ for $m$ in $(\ref{eq:267})$ gives
\begin{equation*}
\sum_{m,n = 0}^\infty  P'_0(m-1+n, m-1-n) x^{m} t^n = \frac{x-t}{1-x-t},
\end{equation*}
which is the same as (\ref{eq:206}) because
\[B(m+1,n)=P'_0(m+n, m-n) \textrm{ and } B'(m,n)=B(m,n) \textrm{ for } 0 \le n \le m. \]

A more interesting case is $T= \{-1,0,1,2\}$ where $K=2$. As we mentioned before, in this case we can easily compute the power series $g$ using the quadratic formula. Then the functional equation for $g$ in Theorem \ref{T:241} becomes $t^3g(t)^2+(1+c_1t)g(t)-1=0$, so
        \begin{equation} \label{eq:252}
         g(t)=  \frac{-1-c_1t + \sqrt{1 + 2c_1t + {c_1}^2 t^2 + 4t^3}}{2t^3},
         \end{equation}
         where the sign of the square root is determined because $g$ is a power series.

Now as a special case we are going to prove the Niederhausen and Sullivan conjecture and give another proof of Theorem \ref{T:221}. For the conjecture we set $h=0$, $c_0=1$, and $c_1=1$ because the paths start at $(0,0)$ and each path with steps $(1,i)$ where $i \in T=\{-1,0,1,2 \}$ is counted once. Then by $(\ref{eq:252})$ we have the power series
\[ g(t) = \frac{-1-t + \sqrt{1+2t+t^2+4t^3}}{2t^3}.\]
That is,
$g(t)=1-t+{t}^{2}-2\,{t}^{3}+4\,{t}^{4}-7\,{t}^{5}+13\,{t}^{6}-26\,{t}^{7}+
52\,{t}^{8}-104\,{t}^{9}+212\,{t}^{10}+
  \cdots .$

So, by Theorem \ref{T:241} we have
the generating function for the
values of $P'_0(m,n)$
\begin{equation*}
\sum_{m,n = 0}^\infty  P'_0(m, 2m-n)x^mt^n = \frac{3+t- \sqrt{1+2t+t^2+4t^3}}{2}\Big( 1-x(1+t+t^2+t^3)
\Big)^{-1},
\end{equation*}
which is the conjecture \cite{NS} of Niederhausen and Sullivan because $S'(m,n)=P'_0(m,2m-n)$ for $m,n \ge 0$. By $P'_0(m,2m-n)=S(4m-n,2m-n)$ where $2m \ge n$, we conclude $S'(m,n) = S(4m-n,2m-n)$ for $0 \le n \le 2m$. That is, $S'(m,n)$ is the number of paths
from $(0,0)$ to $( 4m-n-1, 2m-n+1)$, with up
$(\nearrow)$ steps $(1,1)$ and down $(\searrow)$ steps $(1,-1)$,
that avoid four consecutive up $(\nearrow)$ steps and never go
below the $x$-axis.

For another proof of Theorem \ref{T:221}, first let us consider a lattice path $R$ from $(0,0)$ to $(m,n)$ (where $n \le 2m$), with unit up $(\uparrow)$ steps
$(0,1)$ and unit right $(\rightarrow)$ steps $(1,0)$, that never
crosses the line $y=2x$. To apply Theorem \ref{T:241} we transform the path $R$ into a path from $(0,0)$ to $(m+n,2m-n)$, with up steps $(1,2)$ and down steps $(1,-1)$, that never goes below the $x$-axis. We replace a unit right $(\rightarrow)$ step $(1,0)$ with a step $(1,2)$ and a unit up $(\uparrow)$ step
$(0,1)$ with a step $(1,-1)$. The transformed path does not go below the $x$-axis because when the path $R$ takes a unit right $(\rightarrow)$ step
$(1,0)$ the $y$-coordinate difference between the path $R$ and the
line $y=2x$ increases by $2$, whereas when the path $R$ takes a up
$(\uparrow)$ step $(0,1)$ the $y$-coordinate difference decreases
by $1$. So, we have
\begin{equation}\label{eq:2255}
 D_2(m,n)=P'_{0}(m+n,2m-n) \textrm{ for } 2m \ge n.
  \end{equation}

In Theorem \ref{T:241} we take the step set $T=\{ -1,2\}$. Then we set $h=0$, $c_0=0$, and $c_1=0$ since paths start at $(0,0)$ and have the step set $T=\{-1, 2\}$. Table \ref{T:229} shows the values of $P'_0(m,n)$.
{ \begin{table}
 \centering
\begin{tabular}{c||rrrrrrrr}
  $3$  & $0$ & $0$ & $0$ & $2$ & $0$ & $0$ & $9$ & $0$ \\
  $2$  & $0$ & $1$ & $0$ & $0$ & $3$ & $0$ & $0$ & $12$ \\
  $1$  & $0$ & $0$ & $1$ & $0$ & $0$ & $3$ & $0$ & $0$  \\
  $0$  & $1$ & $0$ & $0$ & $1$ & $0$ & $0$ & $3$ & $0$ \\
  $-1$  & $0$ & $0$ & $0$ & $0$ & $0$ & $0$ & $0$ & $0$  \\
  $-2$  & $0$ & $0$ & $0$ & $0$ & $0$ & $0$ & $0$ & $0$  \\
  $-3$  & $-1$ & $0$ & $0$ & $-1$ & $0$ & $0$ & $-3$ & $0$  \\
  $-4$  & $0$ & $0$ & $0$ & $0$ & $0$ & $0$ & $0$ & $0$ \\
  $-5$  & $0$ & $0$ & $-1$ & $0$ & $0$ & $-3$ & $0$ & $0$  \\
  $-6$  & $1$ & $0$ & $0$ & $1$ & $0$ & $0$ & $3$ & $0$ \\
  $-7$  & $0$ & $-1$ & $0$ & $0$ & $-3$ & $0$ & $0$ & $-12$  \\
  $-8$  & $0$ & $0$ & $2$ & $0$ & $0$ & $6$ & $0$ & $0$ \\
  $-9$  & $-2$ & $0$ & $0$ & $-4$ & $0$ & $0$ & $-15$ & $0$  \\
  $-10$  & $0$ & $3$ & $0$ & $0$ & $9$ & $0$ & $0$ & $36$ \\
  $-11$  & $0$ & $0$ & $-6$ & $0$ & $0$ & $-21$ & $0$ & $0$  \\
  $-12$  & $5$ & $0$ & $0$ & $13$ & $0$ & $0$ & $51$ & $0$ \\
  $-13$  & $0$ & $-9$ & $0$ & $0$ & $-30$ & $0$ & $0$ & $-127$  \\
  $-14$  & $0$ & $0$ & $19$ & $0$ & $0$ & $72$ & $0$ & $0$ \\
   \hline  \hline
  $n/m$  & $0$ & $1$ & $2$ & $3$ & $4$ & $5$ & $6$ & $7$   \\
  \end{tabular}
\caption{The values of $P'_0(m,n)$ }\label{T:229}
\end{table} }

We see from Table \ref{T:204} and Table \ref{T:229} that
\begin{equation}\label{eq:2222}
D_2(m,n)=P'_{0}(m+n,2m-n) \textrm{ for } n \ge 0 \textrm{ and }m+n \ge 0.
\end{equation} We can prove that (\ref{eq:2222}) holds by using the recurrences
(\ref{eq:241}) and (\ref{eq:216}).
There is no combinatorial meaning for $D_2(m,n)$ where $2m<n$ or $n<0$.

  By $(\ref{eq:252})$ we have the power series
\[ g(t) = \frac {-1+\sqrt {1+4t^3}}{2t^3}= \sum_{n=0}^\infty (-1)^nC_n t^{3n}.\]
That is,
$g(t)=1-{t}^{3}+2\,{t}^{6}-5\,{t}^{9}+14\,{t}^{12}-42\,{t}^{15}+132\,{t}^{
18}-429\,{t}^{21}+1430\,{t}^{24} +  \cdots .$
So, by Theorem \ref{T:241} the generating function for the values of $P'_0(m,n)$ is
\begin{equation*}
\sum_{m,n=0}^\infty P'_0(m,2m-n) x^{m}t^{n}
= \Big( 1+ \sum_{n=0}^\infty (-1)^{n+1} C_n t^{3(n+1)}\Big) \Big( 1-x(1+t^3) \Big)^{-1}.
\end{equation*}
Replacing $t$ with $t^{1/3}$ and $3n$ with $n$ gives
\begin{equation*}
\sum_{m,n=0}^\infty P'_0(m,2m-3n) x^{m}t^{n}
 = \Big( 1+ \sum_{n=0}^\infty (-1)^{n+1} C_n t^{n+1}\Big) \Big( 1-x(1+t) \Big)^{-1}.
\end{equation*}
With $D'_2(m,n)$ defined as in (\ref{eq:200}), we know $P'_0(m,2m-3n)=D'_2(m-n,n)$ for $m,n \in \mathbb{N}$. So, we have
 \[ D'_2(-n-1,n+1)=(-1)^{n+1}C_n \textrm{ for all } n \ge 0. \]
From (\ref{eq:2222}) we conclude
  \[ D_2(-n-1,n+1)=(-1)^{n+1}C_n \textrm{ for all } n \ge 0. \]
Also, by $(\ref{eq:2255})$ we deduce
\[ D'_2(m,n)=D_2(m,n) \textrm{ for } 0 \le n \le 2m. \]
That is, $D'_2(m,n)$ is the number of lattice paths from $(0,0)$ to $(m,n)$ (where $0 \le n \le 2m$), with unit up $(\uparrow)$ steps $(0,1)$ and unit right $(\rightarrow)$ steps $(1,0)$, that never cross the line $y=2x$.

For another application of Theorem \ref{T:241} we can find a redundant generating function for $D_p(m,n)$ which was defined in Section \ref{S:203}. To do this let us consider a lattice path $S$ from $(0,0)$ to $(m,n)$ (where $n \le pm$), with unit up $(\uparrow)$ steps
$(0,1)$ and unit right $(\rightarrow)$ steps $(1,0)$, that never
crosses the line $y=px$. To apply Theorem \ref{T:241} we transform the path $S$ into a path from $(0,0)$ to $(m+n,pm-n)$, with up steps $(1,p)$ and down steps $(1,-1)$, that never goes below the $x$-axis. We replace a unit right $(\rightarrow)$ step $(1,0)$ with a step $(1,p)$ and a unit up $(\uparrow)$ step
$(0,1)$ with a step $(1,-1)$. The transformed path does not go below the $x$-axis because when the path $S$ takes a unit right $(\rightarrow)$ step
$(1,0)$ the $y$-coordinate difference between the path $S$ and the
line $y=px$ increases by $p$, whereas when the path $S$ takes a up
$(\uparrow)$ step $(0,1)$ the $y$-coordinate difference decreases
by $1$. So we take the step set $T=\{ -1,p\}$ and
$D_p(m,n)=P'_{0}(m+n,pm-n)$. By Theorem \ref{T:241}, we know that the functional equation for $g(t)$ is
       \begin{equation*}
        g(t)  \sum_{n=1}^p \Big( g(t)t^{p+1} \Big)^{n-1}=1,
         \end{equation*}
         which is equal to $ \sum_{n=1}^p \Big( g(t)t^{p+1} \Big)^{n}= t^{p+1}.$ From this equation we see that $g(t)t^{p+1} $ is a power series in $t^{p+1}$, so we may define a power series $\gamma(t)$ by $\gamma(t^{p+1}) = g(t)t^{p+1}$. Then the equation becomes
                \begin{equation}\label{eq:224}
         \sum_{n=1}^p \gamma(t)^n =t .
         \end{equation}
         That is, $\gamma(t)$ is the compositional inverse of $t+t^2+\cdots +t^p$.
So, by Theorem \ref{T:241} (where $h=0$) we have the generating function for the values of $P'_0(m,n)$ is
\begin{equation}\label{eq:246}
\sum_{m,n=0}^\infty P'_0\big(m,pm-(p+1)n\big) x^{m}t^{n} = \big( 1-g(t)t^{p+1} \big) \big( 1-x(1+t^{p+1}) \big)^{-1}.
\end{equation}

By $D_p(m-n,n)=P'_0\big(m,pm-(p+1)n\big)$ and Theorem \ref{T:241} (where $h=0$), we can deduce the following corollary after replacing $t^{p+1}$ with $t$ and $(p+1)n$ with $n$ in $(\ref{eq:246})$.

\begin{corollary}\label{T:243}
Let $D_p(m,n)$ be the
number of lattice paths from $(0,0)$ to $(m,n)$ (where $n \le
pm$), with unit up $(\uparrow)$ steps $(0,1)$ and unit right
$(\rightarrow)$ steps $(1,0)$, that never cross the line $y=px$.
Define the number $D'_p(m,n)$ for $m,n \in \mathbb{N}$ by
\begin{equation*}
 \sum_{m,\, n =0}^\infty D'_p(m-n,n)x^{m} t^{n}
 =\Big(1-\gamma(t) \Big) \Big(1-x \big(1+t\big)
         \Big)^{-1},
         \end{equation*}
         where the power series $\gamma(t)$ satisfies equation $(\ref{eq:224})$.
         Then we have
         \[D'_p(m,n)=D_p(m,n) \textrm{ for $0 \le n \le pm$}. \]
         \end{corollary}

Up to now we considered a path starting at height $0$, that is, $h=0$. Now we want to consider the more general case of a boundary line $y=px+h$ instead of $y=px$ where $h, p \in \mathbb{P}$. For a fixed $h, p \in
\mathbb{P}$, let $E_{p,h}(m,n)$ be the number of lattice paths
from $(0,0)$ to $(m,n)$ (where $n \le pm +h$), with unit up
$(\uparrow)$ steps $(0,1)$ and unit right $(\rightarrow)$ steps
$(1,0)$, that never cross the line $y=px+h$.

Applying the same transformation as before, we have
\[E_{p,h}(m-n,n)=P'_h \big(m,pm+h-(p+1)n\big) \textrm{ with } T=\{-1,p\}. \]
Therefore, by Theorem \ref{T:241} we can conclude the following corollary after replacing $t^{p+1}$ with $t$ and $(p+1)n$ with $n$.

\begin{corollary}\label{T:244}
For $h,p \in \mathbb{N}$, the number $E'_{p,h}(m,n)$ for $m,n \in \mathbb{N}$ is defined by
\begin{equation*}
 \sum_{m,\, n =0}^\infty E'_{p,h}(m-n,n)x^{m} t^{n}
         =\big(1-\gamma(t)^{h+1} \big) \big(1- x (1+t)\big)^{-1},
\end{equation*}
         where the power series $\gamma(t)$ satisfies equation $(\ref{eq:224})$.
         Then we have
         \[ E'_{p,h}(m,n) = E_{p,h}(m,n) \textrm{ for $0 \le n \le pm +h$}. \]
\end{corollary}
Note that in the case $h=0$, Corollary $\ref{T:244}$ reduces to Corollary $\ref{T:243}$ because $E_{p,0}(m,n)=D_p(m,n)$.


%

\noindent \emph{Keywords: Redundant generating functions, Catalan
numbers}

\end{document}